\documentclass{amsart}
\usepackage{amsfonts,amssymb,amsmath,amsthm}
\usepackage{url}
\usepackage{enumerate}

\newtheorem{thrm}{Theorem}[section]
\newtheorem{lem}[thrm]{Lemma}
\newtheorem{prop}[thrm]{Proposition}
\newtheorem{cor}[thrm]{Corollary}
\theoremstyle{definition}
\newtheorem{definition}[thrm]{Definition}
\newtheorem{remark}[thrm]{Remark}
\numberwithin{equation}{section}

\author{D.B. Killough}
\address[D.B. Killough]{Department of Mathematics, Physics, and Engineering\\
Mount Royal University\\
Calgary, AB, Canada T3E 6K6}
\email{bkillough@mtroyal.ca}

\author{I.F. Putnam}
\address[I.F. Putnam]{
Department of Mathematics and Statistics\\
University of Victoria\\
Victoria, B.C., Canada V8W 3R4}
\email{putnam@math.uvic.ca}
\thanks{First author supposrted in part by an NSERC Scholarship. Second author supported in part by an NSERC Discovery Grant}

\begin{document}
\title[Bowen Measure From Heteroclinic Points]{Bowen Measure From Heteroclinic Points}

\begin{abstract}
We present a new construction of the entropy-maximizing, invariant probability measure on a Smale space (the Bowen measure).  Our construction is based on points that are unstably equivalent to one given point, and stably equivalent to another: heteroclinic points.  The spirit of the construction is similar to Bowen's construction from periodic points, though the techniques are very different.  We also prove results about the growth rate of certain sets of heteroclinic points, and about the stable and unstable components of the Bowen measure.  The approach we take is to prove results through direct computation for the case of a Shift of Finite type, and then use resolving factor maps to extend the results to more general Smale spaces.  
\end{abstract}

\maketitle

\section{Introduction}

A Smale space, as defined by David Ruelle \cite{ruelle78}, is a compact metric space, $X$, together
with a homeomorphism, $\varphi$, which is hyperbolic. 
These include the basic sets of
Smale's Axiom A systems \cite{smale67}. Another special case of great interest are the shifts of finite type \cite{bowen78}, \cite{lindmarcus}
where the space, here usually denoted $\Sigma$, is the path space of a finite directed graph and 
the homeomorphism, $\sigma$, is the left shift.

The structure of $(X,\varphi)$ is such that each point $x$ in $X$ has two local sets associated to it: $X^s(x,\epsilon)$, on which the map $\varphi$ is (uniformly) contracting; and 
 $X^u(x,\epsilon)$, on which the map $\varphi^{-1}$ is contracting.  We call these sets the local stable and unstable sets for $x$.
 Furthermore, $x$ has a neighbourhood, $U(x,\epsilon)$ that is isomorphic to $X^u(x,\epsilon) \times X^s(x,\epsilon)$. In other words, 
the sets $X^u(x,\epsilon)$ and $X^s(x,\epsilon)$ provide a coordinate system for $U(x, \epsilon)$ such that, under application of the map 
$\varphi$, one coordinate contracts, and the other expands. 

The basic axiom for a Smale space is the existence of a map defined on pairs $(x,y)$ in $X \times X$ which are sufficiently close. The image of $(x,y)$ is denoted
$[x, y]$ and is the unique point in $X^{s}(x, \epsilon) \cap X^{u}(y, \epsilon)$. This satisfies a number of identities and, in particular defines a
 homeomorphism from $ X^u(x,\epsilon) \times X^s(x,\epsilon) \rightarrow U(x,\epsilon)$.  

There is also a notion of a global stable (unstable) set for a point $x$, which we denote $X^s(x)$ ($X^u(x)$).  
This is simply the set of all points $y \in X$ such that $d(\varphi^{n}(x), \varphi^{n}(y)) \rightarrow 0$ as $n \rightarrow +\infty \ (-\infty)$.  
The collection of sets $\{ X^s(y,\delta) \ | \ y \in X^s(x), \ \delta > 0 \}$ forms a neighbourhood base for a topology on 
$X^s(x)$ that is locally compact and Hausdorff.  This is the topology that we use on $X^s(x)$ (not the relative topology from $X$). 
 There is an analogous topology on $X^u(x)$.
The global stable (unstable) sets partition the Smale space $X$ into equivalence classes.
 In other words, there are three equivalence relations defined on $X$.  We say $x$ and $y$ are \emph{stably equivalent} if $X^s(x) = X^s(y)$, 
\emph{unstably equivalent} if $X^u(x) = X^u(y)$, and \emph{homoclinic} if they are both stably and unstably equivalent. 
 Finally, we say that a point $z$ is a heteroclinic point for the pair $(x,y)$ if $z$ is 
stably equivalent to $x$ and unstably equivalent to $y$ (i.e. $z \in X^s(x) \cap X^u(y)$).

For an irreducible Smale space, $(X,\varphi)$, there is a unique $\varphi$-invariant probability measure maximizing the entropy 
of $\varphi$ \cite{ruellesullivan}, \cite{katokhass}.  This measure is known as the Bowen measure and we denote it by $\mu_X$, 
or when the space is obvious, simply $\mu$.  

In \cite{bowen71}, Bowen constructed the measure of maximum entropy as a limit of measures supported on periodic points. 
Our main goal in this paper is to present an alternative construction in which the Bowen measure is obtained as the limit of measures 
supported on heteroclinic points. The main result is Theorem \ref{BowenHomIrred}, which is proved in section 4.  From our 
construction we are also able to relate the growth rate of certain sets of heteroclinic points to the topological entropy 
of the Smale space.  A similar result concerning the growth rate of homoclinic orbits was proved by Mendoza in \cite{mendoza89}, 
using different techniques.  

\section{Main Results}\label{BowenHeteroSec}

In \cite{bowen71} the unique entropy maximizing $\varphi$-invariant probability measure is constructed as the weak-$\ast$ limit of 
the sequence $\mu_n$, where $\mu_n$ is defined as follows.  Let $S_n = \cup_1^nPer_n(X,\varphi)$ then
$$
\mu_n = \frac{1}{\#S_n}\sum_{z \in S_n}\delta_z,
$$
where $\delta_z$ is the point mass at $z$. In our construction we use points which are heteroclinic to a given pair of points instead of periodic points.  It is worth noting 
that in Bowen's construction each $\mu_n$ is a $\varphi$-invariant probability measure.  In our case, the measures constructed are 
not $\varphi$-invariant, but in the limit we recover $\varphi$-invariance. 


\begin{definition} \label{HomMeas}
Let $(X,\varphi)$ be a mixing Smale space, $x,\ y \in X$, $B \subset X^u(x)$ and $C \subset X^s(y)$ open with compact closure. For each positive integer $k$, we define 
$$
h_{B,C}^k = \varphi^{k}(B) \cap \varphi^{-k}(C)
$$
and the measure
$$
\mu_{B,C}^k = \frac{1}{\#h_{B,C}^k}\sum_{z\in h_{B,C}^k}\delta_z.
$$
\end{definition}
\begin{remark}
\begin{itemize}
\item As $X^u(x)$ and $X^s(y)$ intersect transversally and $\varphi^{k}(B)$ and $\varphi^{-k}(C)$ have compact closure for each
 $k$, $\# h_{B,C}^k$ is finite for each $k$.
\item  $h_{B,C}^k$ may be empty, and hence $\mu_{B,C}^k$ may not be well defined for some positive integers $k$. However, 
for given $B,\ C$ there exists a $K$ such that for all $k>K$, $\mu_{B,C}^k$ is well defined. Since we will be 
interested in the (weak-$\ast$) limit of these measures as $k \rightarrow \infty$ we will not be concerned with the finite number of $k$'s for which our definition is not valid.
\end{itemize}
\end{remark}

We have the following result relating the growth of the heteroclinic sets $h^k_{B,C}$ to the topological entropy.

\begin{thrm}\label{ratiolimitSmSp}
Let $(X, \varphi)$ be a mixing Smale space, $B$, $C$ as in Defn. \ref{HomMeas}. Then we have
$$
\lim_{k \rightarrow \infty} \lambda^{-2k} \# h^k_{B,C} = \mu^u(B)\mu^s(C),
$$
where $\log(\lambda) = h(X, \varphi)$ is the topological entropy of $(X,\varphi)$.
In consequence, we also have 
$$
\lim_{k \rightarrow \infty} \frac{\log(\#h^k_{B,C})}{2k} = h(X, \varphi).
$$
\end{thrm}

\begin{thrm}\label{BowenHomMix}
Let $(X,\varphi)$ be a mixing Smale space, and let $\mu^k_{B,C}$ be as in Defn. \ref{HomMeas}. For each continuous function $f: X \rightarrow \mathbb{C}$ we have
$$
\lim_{k \rightarrow \infty} \int_X f d\mu_{B,C}^k = \int_X f d\mu_X,
$$
where $\mu_X$ is the Bowen measure. In other words $\mu_{B,C}^k \rightarrow \mu_X$ in the weak-$\ast$ topology.
\end{thrm}

Now suppose $(X,\varphi)$ is an irreducible Smale space (not necessarily mixing). 
 By Smale's spectral decomposition\cite{smale67} we can find a partition of $X$ into pairwise disjoint clopen subsets, $X_{1}, X_{2}, \ldots, X_{I}$ such that
$\varphi(X_{i}) = X_{i+1}$ (with the indices interpreted modulo $I$) and $\varphi^{I} | X_{i}$ mixing, for each $i$.

\begin{definition}\label{HomMeasIrred}
With the notation as above, let $x, y$ be in the same component, $X_{i_{0}}$, of $X$ and let
 $B \subset X^{u}(x)$ and $C \subset X^{s}(y)$ be open with compact closures. For each $k$, we define 
\[
 h^{k}_{B,C} =  \cup_{i=0}^{I-1} ( \varphi^{kI+i}(B)  \cap   \varphi^{-kI+i}(C) ) 
\]
and the measure
\[
 \mu^{k}_{B,C} = \frac{1}{\#  h^{k}_{B,C} } \sum_{z \in  h^{k}_{B,C} } \delta_{z}.
\]
\end{definition}

\begin{remark}
\begin{itemize}
\item The same remark concerning $ h^{k}_{B,C}$ being empty as before applies.
\item In the case that $(X, \varphi)$ is mixing (and $I=1$), this clearly reduces to the same definition as before.
\end{itemize}
\end{remark}

With this extended definition, the analogous results as stated above for the mixing case also hold in the irreducible case.

\begin{thrm}\label{ratiolimitSmSpIrred}
Let $(X, \varphi)$ be an irreducible Smale space, $B$, $C$ as in Defn. \ref{HomMeasIrred}. Then we have
$$
\lim_{k \rightarrow \infty} \lambda^{-2kI} \# h^k_{B,C} = I \mu^u(B)\mu^s(C),
$$
where $\log(\lambda) = h(X, \varphi)$ is the topological entropy of $(X,\varphi)$.
In consequence, we also have 
$$
\lim_{k \rightarrow \infty} \frac{\log(\#h^k_{B,C})}{2kI} = h(X, \varphi).
$$
\end{thrm}

\begin{remark}
Theorem 3.1 in \cite{mendoza89} is essentially this result, replacing $h^k_{B,C}$ with $\varphi^k(h^k_{B,C})$ in the case that the heteroclinic points happen to be homoclinic points.
\end{remark}

\begin{thrm}\label{BowenHomIrred}
Let $(X,\varphi)$ be an irreducible Smale space, and let $\mu^k_{B,C}$ be as in Defn. \ref{HomMeasIrred}. For each continuous function $f: X \rightarrow \mathbb{C}$ we have
$$
\lim_{k \rightarrow \infty} \int_X f d\mu_{B,C}^k = \int_X f d\mu_X,
$$
where $\mu_X$ is the Bowen measure. In other words $\mu_{B,C}^k \rightarrow \mu_X$ in the weak-$\ast$ topology.
\end{thrm}

\section{Resolving Factor Maps and the Bowen Measure}

It was shown in \cite{ruellesullivan} that, if a small subset of $X$ is written as a product, then
the Bowen measure on this set can be written 
 as a product measure. We will describe this result more precisely below.
However, this gives us a useful way of dealing with the Bowen measure. We will actually provide a 
new proof of the result, but along the way, we will also see how this product decomposition
is preserved under resolving maps.

First, let us give a more precise description of the product decomposition of the Bowen measure.
For each $x$ in $X$, we wish to have measures $\mu^{s,x}$ and $\mu^{u,x}$ 
defined on $X^{s}(x)$ and $X^{u}(x)$, respectively.
Secondly, the measure $\mu^{s,x}$  depends only on the stable 
equivalence class of $x$; that is, if $y$ is in $X^{s}(x)$, then 
$\mu^{s,x} = \mu^{s,y}$. (Put another way, we should be writing 
$\mu^{s,X^{s}(x)}$, but that notation is rather too clumsy.)
A similar statement holds for $\mu^{u,x}$. 
It is worth noting that these measures are not finite, but are regular Borel measures.
Moreover, these satisfy the following conditions.
\begin{enumerate}
\item For all $x$ in $X$, $\epsilon > 0$ and Borel sets 
$B \subset X^{u}(x, \epsilon)$ and $C \subset X^{s}(x, \epsilon)$, we have 
\[
 \mu([B, C]) = \mu^{u,x}(B)\mu^{s,x}(C)
\]
whenever $\epsilon$ is sufficiently small so that $[B, C]$ is defined.
\item For $x$, $y$ in $X$, $\epsilon > 0$ and  a Borel set $B \subset X^{u}(x, \epsilon)$, we have 
\[
 \mu^{u,y}([B, y])  = \mu^{u,x}(B),
\]
whenever $d(x,y)$ and $\epsilon$ are sufficiently small so that $[B, y]$ is defined.
\item For $x$, $y$ in $X$, $\epsilon > 0$ and  a Borel set $C \subset X^{s}(x, \epsilon)$, we have 
\[
 \mu^{s,y}([y, C])  = \mu^{s,x}(C),
\]
whenever $d(x,y)$ and $\epsilon$ are sufficiently small so that $[y, C]$ is defined.
\item $\mu^{s,\varphi(x)} \circ \varphi = \lambda^{-1} \mu^{s,x}$.
\item $\mu^{u,\varphi(x)} \circ \varphi = \lambda \mu^{u,x}$.
\end{enumerate}
Here $\log(\lambda)$ is the topological entropy of $(X,\varphi)$.

In the case that the Smale Space is  a shift of finite type (SFT), the Bowen measure is the same as the Parry measure. 
We  present a brief description of the Parry measure for a \emph{mixing} SFT, showing the above result 
for this case. 

Let $(\Sigma, \sigma)$ be a mixing SFT, considered as the edge shift on a directed graph $G$ with adjacency matrix $A$.  
See \cite{lindmarcus} for a thorough treatment of SFTs.  $(\Sigma, \sigma)$ is mixing precisely when $A$ is primitive, i.e. when there 
exists $N$ such that, for $n \geq N$ $A^n$ is strictly positive. This allows us to use the consequence of the Perron-Frobenius 
theorem (Thm. 4.5.12 in \cite{lindmarcus}), which says $\lim_{n \rightarrow \infty}\lambda^{-n}A^n = u_ru_l$, where $u_r, u_l$ are 
the right/left Perron-Frobenius eigenvectors of the matrix $A$ normalized so that $u_lu_r = 1$, and $\lambda$ is the Perron-Frobenius 
eigenvalue. This result is critical in the proof of our main result in the case of SFTs.  
Fix $m > N$, vertices $v_i, v_j$ in the graph, and let $\xi$ be a path of length $2m$, indexed 
from $-m+1$ to $m$, originating at $v_i$ and terminating at $v_j$ ($A$ primitive guarantees such a $\xi$ exists). Consider the set
$$
\Sigma_{m,i,j}(\xi) = \{ x \in \Sigma \ | x_k = \xi_k \ \textrm{for} \ -m+1 \leq k \leq m \}.
$$
The collection of such sets, as $m$, $i$, $j$, and $\xi$ vary over all possible values, forms a base for the topology on $\Sigma$. 
 The Parry measure on such a basic set is 
$$
\mu_{\Sigma}(\Sigma_{m,i,j}(\xi)) = \lambda^{-2m}u_l(i)u_r(j).
$$
Fix $x$ in $\Sigma$ and suppose $t(x_m) = v_j$, $i(x_{-l+1}) = v_i$ consider the sets
\begin{eqnarray*}
\Sigma^u(x,2^{-m}) &=& \{ z \in \Sigma \ | \ z_k = x_k \ \forall k \leq m \}  \\
\Sigma^s(y,2^{-l}) &=& \{ z \in \Sigma \ | \ z_k = x_k \ \forall k \geq -l+1 \},
\end{eqnarray*}
These sets form a base for the topology on $\Sigma^u(x)$ (respectively $\Sigma^s(x)$) in a neighbourhood of $x$.  
Suppose now that $\Sigma^u(z,2^{-m}) \subset \Sigma^u(x,\epsilon)$ and $\Sigma^s(y,2^{-l}) \subset \Sigma^s(x,\epsilon)$
Then the stable/unstable components of the Parry measure are
\begin{eqnarray*}
\mu_{\Sigma}^{u,x}(\Sigma^u(z,2^{-m})) &=& \lambda^{-m}u_r(j) \\
\mu_{\Sigma}^{s,x}(\Sigma^s(y,2^{-l})) &=& \lambda^{-l}u_l(i)
\end{eqnarray*}

We verify that these measures do, in fact, satisfy the conditions listed above. Consider the homeomorphism 
$w \mapsto [w,x']$ from $\Sigma^u(x,\epsilon)$ to $\Sigma^u(x',\epsilon')$.  Under this map 
$$
\Sigma^u(z,2_{-m}) \mapsto \{ v \in \Sigma \ |  \ v_k = z_k \ \forall 0 \leq k \leq m, \ v_k = x'_k \ \forall k \leq 0, \} = \Sigma^u([z,x'],2^{-m}).
$$
Now 
$$
\mu_{\Sigma}^{u,x'}(\Sigma^u([z,x'],2^{-m})) = \lambda^{-m}u_r(j) = \mu_{\Sigma}^{u,x}(\Sigma^u(z,2^{-m})).
$$
Similarly, the map $w \mapsto [x',w]$ takes the measure $\mu_{\Sigma}^{s,x}$ to $\mu_{\Sigma}^{s,x'}$.  Now consider
$$
(\mu_{\Sigma}^{u,\sigma(x)}\circ \sigma)(\Sigma^u(z,2^{-m})) = \mu_{\Sigma}^{u,\sigma(x)}(\Sigma^u(\sigma(z),2^{-m+1}))
 = \lambda^{-m+1}u_r(j) = \lambda \mu_{\Sigma}^{u,x}(\Sigma^u(z,2^{-m})).
$$
Similarly, 
$$
(\mu_{\Sigma}^{s,\sigma(x)}\circ \sigma)(\Sigma^s(y,2^{-l})) = \mu_{\Sigma}^{s,\sigma(x)}(\Sigma^s(\sigma(y),2^{-l-1}))
 = \lambda^{-l-1}u_l(i) = \lambda^{-1} \mu_{\Sigma}^{s,x}(\Sigma^s(y,2^{-l})).
$$
We have thus verified that the Parry measure on a mixing SFT has local stable/unstable components that satisfy the
 conditions above (ie those in \cite{ruellesullivan}). In section 3 we re-prove the existence of these stable/unstable 
measures for a general mixing Smale space using the explicit form above in the case of a SFT, and the resolving map results of \cite{putnam05}.

In the case of a SFT, the topological entropy $h(\Sigma, \sigma) = \log(\lambda)$, where $\lambda$ is the Perron-Frobenius 
eigenvalue of the adjacency matrix associated with the SFT.  Similarly, for other Smale spaces $X$ we will write 
$\lambda$ such that $h(X, \varphi) = \log(\lambda)$.  Whenever we are talking about 2 or more Smale spaces, there will be an 
almost one-to-one factor map between them, so the entropies will be equal, hence it will be unnecessary to distinguish which space the $\lambda$ comes from. 

\begin{definition}[Fried \cite{fried87}]
A factor map $\pi: (Y,\psi) \rightarrow (X,\varphi)$ is \emph{$s$-resolving} (\emph{$u$-resolving}) if for every $y \in Y$, 
$\pi|_{Y^s(y)}$ ($\pi|_{Y^u(y)}$ respectively) is injective.
\end{definition}

We will primarily be concerned with \emph{almost one-to-one} resolving factor maps.  A factor map $\pi: (Y,\psi) \rightarrow (X,\varphi)$,
 where $(Y,\psi)$ is irreducible,  is called \emph{almost one-to-one} if there exists $x \in X$ such that $\#\pi^{-1}(x) = 1$.

In \cite{bowen70}, Bowen showed that for an irreducible Smale space, $(X,\varphi)$, there exists an irreducible SFT $(\Sigma, \sigma)$ 
and an almost one-to-one factor map $\pi: \Sigma \rightarrow X$.  Moreover, letting $E = \{x \in X \ | \ \#\pi^{-1}(x) = 1 \}$, Bowen
 showed that $\mu_{\Sigma}(\pi^{-1}(E)) = 1$. In other words, $\pi$ is one-to-one $\mu_{\Sigma}$-a.e.  It follows that for any Borel
 set $B \subset X$, $\mu_X(B) = \mu_{\Sigma}(\pi^{-1}(B))$ (Theorem 34 in \cite{bowen70}).

In Cor. 1.4 of \cite{putnam05}, the second author showed that the factor map, $\pi$, can be realized as the composition of two 
resolving factor maps. In other words, given an irreducible Smale space $(X, \varphi)$, we can find a Smale space $(Y, \psi)$, 
a SFT $(\Sigma, \sigma)$, and factor maps $\pi_1: \Sigma \rightarrow Y$, $\pi_2: Y \rightarrow X$ such that
\begin{enumerate}
\item $(\Sigma, \sigma)$ and $(Y, \psi)$ are irreducible,
\item $\pi_1$ and $\pi_2$ are almost one-to-one,
\item $\pi_1$ is $s$-resolving and $\pi_2$ is $u$-resolving.
\end{enumerate}
The Bowen measures on $X$, $Y$ can be obtained from the Bowen measure on $\Sigma$ as follows
\begin{enumerate}
\item for $E \subset Y$ the Bowen measure on $(Y, \psi)$ is $\mu_Y(E) = \mu_{\Sigma}(\pi_1^{-1}(E))$, 
\item for $F \subset X$ the Bowen measure on $(X, \varphi)$ is $\mu_X(F) = \mu_{Y}(\pi_2^{-1}(F)) =\mu_{\Sigma}((\pi_2 \circ \pi_1)^{-1}(F)) $.
\end{enumerate}
This requires only that $\pi_1$, $\pi_2$ be almost one-to-one factor maps, not that they are resolving.
We now wish to define the measures on the stable and unstable equivalence classes in $(Y, \psi)$ and $(X, \varphi)$, from $\mu_{\Sigma}^s$, $\mu_{\Sigma}^u$, $\pi_1$, and $\pi_2$.  In this case, it is not enough that the factor maps are almost one-to-one, resolving plays an important role in what follows. We begin by stating the following result which is proved in \cite{putnam00}.

\begin{prop}
Let $(Y, \psi)$ and $(X, \varphi)$ be irreducible Smale spaces, and $\pi: Y \rightarrow X$ be an almost one-to-one $u$-resolving factor map.  If $x \! \in \! X$ with $\pi^{-1}(x) \! = \! \{y_1, y_2, \ldots, y_n \}$ then 
$$
\pi^{-1}(X^u(x)) =  \bigcup_{i = 1}^n  Y^u(y_i),
$$
and the union is disjoint.  Moreover, using the topologies from the introduction, for each $1 \leq i \leq n$
$$
\pi|_{Y^u(y_i)}: Y^u(y_i) \rightarrow X^u(x)
$$
is a homeomorphism.
\end{prop} 



\begin{lem} \label{resolve1to1}
Let $(Y,\psi)$ and $(X,\varphi)$ be irreducible Smale spaces, and $\pi: Y \rightarrow X$ be an almost one-to-one $u$-resolving factor map. Fix $y \in Y$, the set $\{y' \in Y^s(y) \ | \ \pi(y') = \pi(\tilde{y})  \ \textrm{for some} \ \tilde{y} \neq y' \}$ has $\mu_Y^s$ measure zero.  In other words, $\pi|_{Y^s(y)}$ is one-to-one $\mu_Y^s$ almost everywhere.
\end{lem}
\begin{proof}
As $Y$ is compact, we may cover $Y$ with a finite number of sets of the form $U_i = [Y^u(z_i,\delta_i),Y^s(z_i,\delta_i)]$. Fix $U_i$ and $y \in U_i$, let $B_i = [Y^u(z_i,\delta_i),y]$, $C_i = [y,Y^s(z_i,\delta_i)]$, so we can write $U_i = [B_i,C_i]$. 

Let $S_i = \{y' \in Y^s(y,\epsilon) \cap C_i \ | \ \pi(y') = \pi(\tilde{y})  \ \textrm{for some} \ \tilde{y} \neq y' \}$. Since $\pi$ is $u$-resolving, the set $U_i \cap \{y' \in Y \ | \ \pi(y') = \pi(z)  \ \textrm{for some} \ z \neq y' \} = [B_i,S_i]$. Now, we know that $\pi$ is 1-to-1 $\mu_Y$ almost everywhere, so
$$
0 = \mu_Y([B_i,S_i]) = \mu_Y^{u,y}(B_i)\mu_Y^{s,y}(S_i).
$$
We also know that 
$$
0 \neq \mu_Y(U_i) = \mu_Y([B_i,C_i]) = \mu_Y^{u,y}(B_i)\mu_Y^{s,y}(C_i).
$$
So $\mu_Y^{u,y}(B_i) \neq 0$ and thus $\mu_Y^{s,y}(S_i) = 0$. The conclusion follows. 
\end{proof}

Note that the analogous result with an $s$-resolving map and $\mu_Y^u$ also holds.

\begin{lem}\label{finitepreimage}
Let $\pi: (Y,\psi) \rightarrow (X,\varphi)$ be an almost one-to-one $u$-resolving factor map. There exists a constant $M$ such that if
 $x \in X$, $C \subset X^s(x)$ open with compact closure, and $C' = \pi^{-1}(C)$, then $C' = \cup_1^mC'_i$ where the union is disjoint,
 $m \leq M$ and $C'_i \subset Y^s(y_i)$ for some $y_i \in Y$.
\end{lem}

\begin{proof}
Let $\epsilon_{\pi}$ be as in Lemma 3.2 of \cite{putnam00}, and cover $Y$ with $\epsilon_{\pi}$-balls. Let $M$ be the minimum number 
of sets in such a cover, and label them $\{U_i\}_1^M$. Now choose $y_i \in C' \cap U_i$ for each $i$ such that this set is non-empty, 
and let $\tilde{C}_i = C' \cap U_i$. We show that $\tilde{C}_i \subset Y^s(y_i)$.  Suppose $y \in \tilde{C}_i$, then 
$[y_i,y] \in Y^s(y_i)\cap U_i$. Now $d(y,y_i) < \epsilon_{\pi}$, so from Lemma 3.2 of \cite{putnam00}, $\pi([y_i,y]) = [\pi(y_i),\pi(y)]$. 
 $\pi(y), \pi(y_i)$ are both in $C \subset X^s(x)$ so they are stably equivalent. Therefore $[\pi(y_i),\pi(y)] = \pi(y)$.  We have that $\pi([y_i,y]) = \pi(y)$, so by Lemma 3.3 of \cite{putnam00}, $y$ and $[y_i,y]$ are stably equivalent.  $[y_i,y]$ is also stably equivalent to $y_i$, so $y$ is stably equivalent to $y_i$ and we have shown that $\tilde{C}_i \subset Y^s(y_i)$.  

At this point the sets $\tilde{C}_i$ may not be pairwise disjoint. Suppose $y \in \tilde{C}_i \cap \tilde{C}_j$. Then 
from above $y \in Y^s(y_i)\cap Y^s(y_j)$ and thus $Y^s(y_i)= Y^s(y_j)$. We then replace these two sets with the union $\tilde{C}_i \cup \tilde{C}_j$. 
 Similarly, if $y$ is in $n$ of the $\tilde{C_i}$ sets, then all $n$ of these sets are contained in one unstable set and we union them all together.  
In this manner we arrive at a collection of sets $\{C'_i\}_1^m$ such that $m \leq M$, the $C'_i$'s are pairwise disjoint, $\cup_1^m C'_i = C'$,
 and $C'_i \subset Y^s(y_i)$ for some $y_i$.
\end{proof}

\begin{definition}\label{SUmeasures}
Let $(Y,\psi)$ and $(X,\varphi)$ be irreducible Smale spaces, and $\pi: Y \rightarrow X$ an almost one-to-one $u$-resolving factor map. 
 Let $x \in X$ and $X^s(x_1,\delta) \subset X^s(x,\epsilon)$, $X^u(x_2,\delta) \subset X^u(x,\epsilon)$.  Fix $y \in Y$ and $U(y) \subset Y^u(y)$ 
such that $\pi(y) = x_2$, $\pi(U(y)) = X^u(x_2,\delta)$.
As in Lemma \ref{finitepreimage} find points $\{y_i\}_1^m \in Y$ and sets $C_i' \subset Y^s(y_i)$ such that $\pi^{-1}(X^s(x_1,\delta)) = \bigcup_1^mC_i'$.
Define measures on $X^s(x)$, $X^u(x)$ locally by
\begin{eqnarray*}
\mu_{X}^{s,x} (X^s(x_1,\delta))&=& \sum_1^m\mu_Y^{s,y_i}(C_i') \\
\mu_{X}^{u,x} (X^u(x_2,\delta))&=& \mu_Y^{u,y}(U(y))
\end{eqnarray*}
\end{definition}

\begin{remark} 
We have stated Defn. \ref{SUmeasures} in terms of an almost one-to-one $u$-resolving factor map.  Given two Smale spaces and an almost 
one-to-one $s$-resolving factor map, we would make the analogous definition, interchanging roles of stable and unstable sets. 
\end{remark}


\begin{prop} \label{resolvmeas}
Let $(Y, \psi)$, $(X,\varphi)$ be irreducible Smale spaces, and $\pi: Y \rightarrow X$ be an almost one-to-one resolving factor map
 ($s$, or $u$-resolving). Suppose $Y$ has local stable/unstable measures $\mu_Y^{s,y}$/$\mu_Y^{u,y}$ that satisfy the conditions 
above, then the measures on $X$ defined in Defn. \ref{SUmeasures} also satisfy the conditions for stable/unstable
 components of the Bowen measure as above.
\end{prop}
\begin{proof}
We prove the result in the case that $\pi$ is $u$-resolving.  The $s$-resolving case is completely analogous.
Let $x \in X$ and let $C = X^s(x_1,\delta) \subset X^s(x,\epsilon)$, $B = X^u(x_2,\delta) \subset X^u(x,\epsilon)$.  Fix $y \in Y$ 
and $U(y) \subset Y^u(y)$ such that $\pi(y) = x_2$, $\pi(U(y)) = X^u(x_2,\delta) = B$.
We need to show
\begin{enumerate}
\item $\mu_X([B,C]) = \mu_X^{u,x}(B)\mu_X^{s,x}(C)$
\item For $z$ close to $x$, $\mu_X^{u,x}(B) = \mu_X^{u,z}([B,z])$
\item $\mu_X^{u,\varphi(x)}(\varphi(B)) = \lambda \mu_X^{u,x}(B)$
\item For $z$ close to $x$, $\mu_X^{s,x}(C) = \mu_X^{s,z}([z,C])$
\item $\mu_X^{s,\varphi(x)}(\varphi(C)) = \lambda^{-1} \mu_X^{s,x}(C)$.
\end{enumerate}
We will prove item 2 fist, as we will use this result in the proof of item 1.
\begin{enumerate}
\item[P2]
We can find $y' \in \pi^{-1}(z)$ such that $y'$ is `close' to $y$. Then, 
$$
\mu_X^{u,x}(B) = \mu_Y^{u,y}(U(y)) = \mu_Y^{u,y'}([U(y),y']) = \mu_X^{u,z}(\pi([U(y),y']))
$$
but $\pi([U(y),y']) = [\pi(U(y)),\pi(y')]$ and $\pi(U(y)) = B$, $\pi(y')=z$ so we have 
$$
 \mu_X^{u,x}(B) = \mu_X^{u,z}(\pi([U(y),y'])) = \mu_X^{u,z}([\pi(U(y)),\pi(y')]) = \mu_X^{u,z}([B,z]).
$$
\item[P1]
Since $C \subset X^s(x,\epsilon)$ is open with compact closure, by lemma \ref{finitepreimage} we can write
$$
\pi^{-1}(C) = \bigcup_{i=1}^m C_i'
$$
where $C_i' \subset Y^s(y_i)$ for some $y_i \in Y$. Moreover, we write each $C_i'$ as a disjoint union of finitely many sets
$$
C_i' = \bigcup_{j=1}^{k_i}C_{ij}'
$$
where $C_{ij}' \subset Y^s(y_{ij}, \epsilon_Y/2)$, and $y_{ij} \in C_{ij}'$.  Let $x_{ij} = \pi(y_{ij})$, and let 
$B_{ij} = [B,x_{ij}]$.  Let $B_{ij}' \subset Y^u(y_{ij})$ be such that $\pi: B_{ij}' \rightarrow B_{ij}$ is a homeomorphism.
We can then write
$$
[B,C] = \pi \left( \bigcup_{i,j}[B_{ij}',C_{ij}'] \right).
$$
So
$$
\mu_X([B,C]) = \mu_Y \left( \bigcup_{i,j}[B_{ij}',C_{ij}'] \right) = \sum_{i,j}\mu_Y([B_{ij}',C_{ij}']) = \sum_{i,j}\mu_Y^{u,y_{ij}}(B_{ij}')\mu_Y^{s,y_{ij}}(C_{ij}').
$$
Now $\mu_Y^{u,y_{ij}}(B_{ij}') = \mu_X^{u,x_{ij}}(B_{ij}) = \mu_X^{u,x}(B)$ for all $i,j$ (by part 1), so we have
$$
\mu_X([B,C]) = \mu_X^{u,x}(B)\sum_{i,j}\mu_Y^{s,y_{ij}}(C_{ij}') = \mu_X^{u,x}(B)\sum_{i}\mu_Y^{s,y_{i}}(C_{i}') = \mu_X^{u,x}(A)\mu_X^{s,x}(B)
$$
\item[P3]
$$
\mu_X^{u,\varphi(x)}(\varphi(B)) = \mu_Y^{u, \psi(y)}(\psi(U(y))) = \lambda \mu_Y^{u,y}(U(y)) = \lambda \mu_X^{u,x}(B).
$$
\item[P4]
We can find $y' \in \pi^{-1}(z)$ such that $y'\in Y^u(y,\epsilon)$. Let $x_{ij}$, $y_{ij}$, $C_i'$, $C_{ij}'$ be as in part 2. 
Let $C_{ij} = \pi(C_{ij}')$, $z_{ij} = [z,x_{ij}]$, $C(z)_{ij} = [z,C_{ij}]$, $y_{ij}' \in \pi^{-1}(z_{ij})$ s.t. $y_{ij}' \in Y^u(y_{ij},\epsilon)$ and $\tilde{C}_{ij}' = [y_{ij}',C_{ij}']$. Then 
$z_{ij} = \pi(y_{ij}')$, $C_{ij} = \pi(C_{ij}')$, $\pi(\tilde{C}_{ij}') = [z_{ij},C_{ij}]$, and $\cup C(z)_{ij} = [z,\cup C_{ij}] = [z,C]$ so
\begin{equation*}
\begin{split}
\mu_X^{s,x}(C) = \sum_i \mu_Y^{s,y_i}(C_i') &= \sum_{i,j}{\mu_Y^{s,y_{ij}}(C_{ij}')} \\
&= \sum_{i,j}{\mu_Y^{s,y_{ij}'}(\tilde{C}_{ij}')} = \mu_X^{s,z}(\cup_{ij}C(z)_{ij}) = \mu_X^{s,z}([z,C])
\end{split}
\end{equation*}

\item[P5] 
$$
\mu_X^{s,\varphi(x)}(\varphi(C)) = \sum_i\mu_Y^{s,\psi(y_i)}(\psi(C_i')) = \sum_i\lambda_{-1}\mu_Y^{s,y_i}(C_i') 
= \lambda^{-1}\mu_X^{s,x}(C)
$$
\end{enumerate} 
\end{proof}

\begin{cor}\label{SUMeasExist}
Let $(X,\varphi)$ be an irreducible Smale space. Then there exist local stable/unstable measures that satisfy the conditions outlined 
at the beginning of this section.
\end{cor}

\begin{proof}
As in Cor. 1.4 in \cite{putnam05}, for the irreducible Smale space $(X,\varphi)$ we can find another irreducible Smale space $(Y,\psi)$
 and an irreducible SFT $(\Sigma,\sigma)$, as well as almost 1-to-1 factor maps $\pi_1:\Sigma \rightarrow Y$, $\pi_2:Y \rightarrow X$
 such that $\pi_1$ is $s$-resolving and $\pi_2$ is $u$-resolving.  The conclusion then follows from the explicit form of the Parry 
measure on $\Sigma$ and 2 applications of Prop. \ref{resolvmeas}.
\end{proof}

\section{Proof of Main Result}

To prove Theorem \ref{BowenHomMix} we first establish the result for a mixing SFT and use the machinery of resolving maps to obtain the more general result.  

\begin{prop}\label{BowenHomSFTMix}
Let $(\Sigma, \sigma)$ be a mixing SFT. Fix $x,y \in \Sigma$, $n,m \in \mathbb{Z}$ and define 
\begin{eqnarray*}
B &=& \{z \in \Sigma \ | \ z_i = x_i \ \forall i \leq n \} = \Sigma^u_n(x) \subset \Sigma^u(x,\epsilon_{\Sigma}) \\
C &=& \{z \in \Sigma \ | \ z_i = y_i \ \forall i \geq -m+1 \} = \Sigma^s_m(y) \subset \Sigma^s(y,\epsilon_{\Sigma}). 
\end{eqnarray*}
For each function $f \in C(\Sigma)$ we have 
$$
\lim_{k \rightarrow \infty} \int_{\Sigma}{f d\mu^k_{B,C}} = \int_{\Sigma}{f d\mu_{\Sigma}}.
$$
In other words, $\mu^k_{B,C} \rightarrow \mu_{\Sigma}$ in the weak-$\ast$ topology.
\end{prop}
\begin{proof} It suffices to prove the result for a function of the form $e_l(\xi) = \chi_{E_l(\xi)}$. Where $E_l(\xi) = \Sigma_{l,i',j'}(\xi)$. 
Now for $k \geq max \{n+l,m+l \}$
$$
{\int_{\Sigma}{e_l(\xi) d\mu^k_{B,C}}} = \mu^k_{B,C}(E_l(\xi)) = \frac{\# \left( E_l(\xi)\cap h^k_{B,C} \right)}{\# h^k_{B,C}}.
$$
The number of points in $E_l(\xi)\cap h^k_{B,C}$ is equal to the number of paths of length $k-(n+l)$ from $t(\sigma^k(x)_{-k+n}) = t(x_n) = v_i$ to $i(\xi_{-l+1}) = v_{i'}$, which equals $A^{k-(n+l)}_{ii'}$, times the number of paths of length $k-(n+l)$ from $t(\xi_{l}) = v_{j'}$ to $i(\sigma^{-k}(y)_{k-m+1}) = i(y_{-m+1}) = v_j$, or $A^{k-(m+l)}_{j'j}$. The number of points in $h^k_{B,C}$ is the number of paths from $t(\sigma^k(x)_{-k+n}) = t(x_n) = v_i$ to $i(\sigma^{-k}(y)_{k-m+1}) = i(y_{-m+1}) = v_j$, or $A^{2k-(n+m)}_{ij}$.
We therefore have
$$
{\int_{\Sigma}{e_l(\xi) d\mu^k_{B,C}}} = \frac{A^{k-(n+l)}_{ii'}A^{k-(m+l)}_{j'j}}{A^{2k-(n+m)}_{ij}},
$$
and
\begin{equation*}
\begin{split}
\lim_{k \rightarrow \infty}& {\int_{\Sigma}{e_l(\xi) d\mu^k_{B,C}}} \\
&= \lim_{k \rightarrow \infty} {\frac{A^{k-(n+l)}_{ii'}A^{k-(m+l)}_{j'j}}{A^{2k-(n+m)}_{ij}}} \\
 &= \lim_{k \rightarrow \infty} \frac{e_iA^{k-(n+l)}e_{i'} e_{j'}A^{k-(m+l)}e_{j}}{e_{i}A^{2k-(n+m)}e_{j}} \\
 &= \lambda^{-2l} \frac{e_i \lim_{k}(\lambda^{-k+(n+l)}A^{k-(n+l)})e_{i'} e_{j'}\lim_{k}(\lambda^{-k+(m+l)}A^{k-(m+l)})e_{j}}{e_{i}\lim_{k}(\lambda^{-2k+(n+m)}A^{2k-(n+m)})e_{j}} \\ 
 &= \lambda^{-2l}\frac{e_i(u_ru_l)e_{i'} e_{j'}(u_ru_l)e_{j}}{e_i(u_ru_l)e_{j}}  \quad \textrm{(by Thm. 4.5.12 in \cite{lindmarcus})}\\
 &= \lambda^{-2l}\frac{u_r(i)u_l(i')u_r(j')(u_l(j)}{u_r(i)u_l(j)} \\
 &= \lambda^{-2l}u_l(i')u_r(j') \\
 &= \mu_{\Sigma}(\Sigma_{l,i',j'}(\xi)) \\
 &= \int_{\Sigma}{e_l(\xi) d\mu_{\Sigma}}.
\end{split}
\end{equation*}  
\end{proof}

In the above, the choice of the sets $B$, $C$, is limited to certain basic sets. We now wish to extend this result to open sets with compact closure $B' \subset \Sigma^u(x)$, $C' \subset \Sigma^s(x)$. To do this we will first need the following lemmas.

\begin{lem}\label{ratiolimit1}
Let $(\Sigma, \sigma)$ be a mixing SFT.  Fix $x,y \in \Sigma$, $n,m \in \mathbb{Z}$ and define 
\begin{eqnarray*}
B &=& \{z \in \Sigma \ | \ z_i = x_i \ \forall i \leq n \} = \Sigma^u_n(x) \subset \Sigma^u(x,\epsilon_{\Sigma}) \\
C &=& \{z \in \Sigma \ | \ z_i = y_i \ \forall i \geq -m+1 \} = \Sigma^s_m(y) \subset \Sigma^s(y,\epsilon_{\Sigma}). 
\end{eqnarray*}
Then 
$$
\lim_{k \rightarrow \infty} \lambda^{-2k} \#h^k_{B,C} = \mu^u(B)\mu^s(C),
$$
where $\log(\lambda) = h(\Sigma, \sigma)$.
\end{lem}

\begin{proof}
Let $t(x_n) = v_i$ and $i(y_{-m+1}) = v_j$,
We then have
\begin{eqnarray*}
\lim_{k \rightarrow \infty}\lambda^{-2k}\#h^k_{B,C} &=& \lim_{k \rightarrow \infty}\lambda^{-2k}A_{ij}^{2k-(n+m)} \\
&=& \lambda^{-(n+m)}\lim_{k \rightarrow \infty}\lambda^{-2k+n+m}e_iA^{2k-(n+m)}e_j \\
&=& \lambda^{-(n+m)}e_iu_ru_le_j \\
&=& \lambda^{-n}u_r(i)\lambda^{-m}u_l(j) \\
&=& \mu^u(B)\mu^s(C).
\end{eqnarray*}
\end{proof}

\begin{lem}\label{ratiolimit2}
Let $B\subset \Sigma^u(x)$, $C \subset \Sigma^s(y)$ be open and compact. Then 
$$
\lim_{k \rightarrow \infty} \lambda^{-2k}\# h^k_{B,C} = \mu^u_{\Sigma}(B)\mu^s_{\Sigma}(C),
$$
where $\log(\lambda) = h(\Sigma, \sigma)$.
\end{lem}

\begin{proof}
If $B$ and $C$ are clopen, then each is a finite disjoint union of cylinder sets of the form considered in Lemma \ref{ratiolimit1}. Let
$$
B = \sum_{i=1}^nB_i, \quad C = \sum_{i=1}^mC_i,
$$
then for fixed $k$ the $h^k_{B_i,C_j}$ are pairwise disjoint and $\cup_{i,j}h^k_{B_i,C_j} = h^k_{B,C}$. Using Lemma \ref{ratiolimit1} we can now write
\begin{eqnarray*}
\lim_{k \rightarrow \infty} \lambda^{-2k}\# h^k_{B,C} &=& \lim_{k \rightarrow \infty} \sum_{i,j}\lambda^{-2k}\# h^k_{B_i,C_j} \\
&=& \sum_{i,j}\lim_{k \rightarrow \infty}\lambda^{-2k}\# h^k_{B_i,C_j} \\
&=& \sum_{i,j} \mu^u(B_i)\mu^s(C_j) \\
&=& \mu^u(B)\mu^s(C).
\end{eqnarray*}
\end{proof}

\begin{lem}\label{ratiolimit3}
Let $B\subset \Sigma^u(x)$, $C \subset \Sigma^s(y)$ be open with compact closure. Then 
$$
\lim_{k \rightarrow \infty} \lambda^{-2k}\# h^k_{B,C} = \mu^u_{\Sigma}(B)\mu^s_{\Sigma}(C),
$$
where $\log(\lambda) = h(\Sigma, \sigma)$.
\end{lem}

\begin{proof}
Fix $\epsilon > 0$
We can find sets $B_1 \subseteq B \subseteq B_2 \subset \Sigma^u(x)$ and $C_1 \subseteq C \subseteq C_2 \subset \Sigma^s(y)$ such that $B_1,\ B_2,\ C_1$ and $C_2$ are compact and open and 
$$
\mu^u(B_2)\mu^s(C_2) - \epsilon < \mu^u(B)\mu^s(C) < \mu^u(B_1)\mu^s(C_1) + \epsilon.
$$
Notice that $\# h^k_{B_1,C_1} \leq \# h^k_{B,C} \leq \# h^k_{B_2,C_2}$, so
\begin{eqnarray*}
\mu^u(B)\mu^s(C) - \epsilon &<& \mu^u(B_1)\mu^s(C_1) \\
&=& \lim_{k \rightarrow \infty} \lambda^{-2k}\# h^k_{B_1,C_1} \\
& \leq & \liminf_{k \rightarrow \infty} \lambda^{-2k}\# h^k_{B,C}
\end{eqnarray*}
and
\begin{eqnarray*}
\mu^u(B)\mu^s(C) + \epsilon &>& \mu^u(B_2)\mu^s(C_2) \\
&=& \lim_{k \rightarrow \infty} \lambda^{-2k}\# h^k_{B_2,C_2} \\
& \geq & \limsup_{k \rightarrow \infty} \lambda^{-2k}\# h^k_{B,C}.
\end{eqnarray*}
As this hold for all $\epsilon >0$ we have 
$$
\limsup_{k \rightarrow \infty} \lambda^{-2k}\# h^k_{B,C} \leq \mu^u(B)\mu^s(C) \leq \liminf_{k \rightarrow \infty} \lambda^{-2k}\# h^k_{B,C}
$$
and hence
$$
\lim_{k \rightarrow \infty} \lambda^{-2k}\# h^k_{B,C} = \mu^u_{\Sigma}(B)\mu^s_{\Sigma}(C).
$$
\end{proof}

We are now ready to prove the more general version of Prop. \ref{BowenHomSFTMix}.

\begin{prop}\label{BowenHomSFTMix2}
The result of Prop. \ref{BowenHomSFTMix} holds with $B \subset \Sigma^u(x)$, $C \subset \Sigma^s(y)$ open with compact closure.
\end{prop}
\begin{proof}
We can write 
$$
B = \bigcup_i B_i, \ \ C = \bigcup_j C_i
$$
where each $B_i$, $C_i$ is of the form considered in Prop. \ref{BowenHomSFTMix}, and the unions are disjoint. For brevity we write
$$
h^k = h^k_{B,C}, \ \ \mu^k = \mu^k_{B,C}
$$
and
$$
h^k_{ij} = h^k_{B_i,C_j}, \ \mu^{k}_{ij} = \mu^k_{B_i,C_j}
$$
Notice that for fixed $k$ the $h^k_{ij}$'s are pairwise disjoint and $\cup_{i,j}h^k_{ij} = h^k$.
We can write
$$
\lim_{k \rightarrow \infty}\int_{\Sigma}fd\mu^k = \lim_{k \rightarrow \infty}\sum_{i,j}\frac{\#h^k_{ij}}{\#h^k}\int_{\Sigma}fd\mu^k_{ij} .
$$
Now let $M = sup_{z\in \Sigma}|f(z)|$, which is finite as $f$ is continuous and $\Sigma$ is compact. For each $k$, $\sum_{i,j}\frac{\#h^k_{ij}}{\#h^k} = 1$ so for any $I \in \mathbb{N}$ we can write
$$
1 = \lim_{k \rightarrow \infty}\sum_{i,j}\frac{\#h^k_{ij}}{\#h^k} = \lim_{k \rightarrow \infty}\sum_{i,j = 1}^I\frac{\#h^k_{ij}}{\#h^k} + \lim_{k \rightarrow \infty}\sum_{i>I,j>I}\frac{\#h^k_{ij}}{\#h^k}.
$$
We also know that
$$
1 = \sum_{i,j}\frac{\mu^u(B_i)\mu^s(C_j)}{\mu^u(B)\mu^s(C)}
$$
and we may choose $I$ large enough so that 
$$
\lim_{k \rightarrow \infty}\sum_{i>I,j>I}\frac{\#h^k_{ij}}{\#h^k} < \frac{\epsilon}{2M},
$$
and
$$
|\sum_{i,j}^I\frac{\mu^u(B_i)\mu^s(C_j)}{\mu^u(B)\mu^s(C)} - 1| < \frac{\epsilon}{2M}.
$$
Using Lemma \ref{ratiolimit3} and Prop. \ref{BowenHomSFTMix} we now have
\begin{eqnarray*}
\left| \lim_{k \rightarrow \infty}\int_{\Sigma}fd\mu^k - \int_{\Sigma}fd\mu\right| \! &=& \left|\lim_{k \rightarrow \infty}\sum_{i,j}\frac{\#h^k_{ij}}{\#h^k}\int_{\Sigma}fd\mu^k_{ij}  - \int_{\Sigma}fd\mu\right| \\
\! &=& \left|\lim_{k \rightarrow \infty}\sum_{i,j}^I\frac{\#h^k_{ij}}{\#h^k}\int_{\Sigma}fd\mu^k_{ij}  + \lim_{k \rightarrow \infty}\sum_{i,j>I}\frac{\#h^k_{ij}}{\#h^k} - \int_{\Sigma}fd\mu\right| \\
\! &=& \left|\sum_{i,j}^I\lim_{k \rightarrow \infty}\frac{\lambda^{-2k}\#h^k_{ij}}{\lambda^{-2k}\#h^k}\int_{\Sigma}fd\mu^k_{ij}  + \lim_{k \rightarrow \infty}\sum_{i,j>I}\frac{\#h^k_{ij}}{\#h^k} - \int_{\Sigma}fd\mu\right| \\
\! &=& \left|\sum_{i,j}^I\frac{\mu^u(B_i)\mu^s(C_j)}{\mu^u(B)\mu^s(C)}\int_{\Sigma}fd\mu + \lim_{k \rightarrow \infty}\sum_{i,j>I}\frac{\#h^k_{ij}}{\#h^k} - \int_{\Sigma}fd\mu\right| \\
\! &\leq & \left|\int_{\Sigma}fd\mu \left(\sum_{i,j}^I\frac{\mu^u(B_i)\mu^s(C_j)}{\mu^u(B)\mu^s(C)} - 1 \right) \right| + \left| \lim_{k \rightarrow \infty}\sum_{i,j>I}\frac{\#h^k_{ij}}{\#h^k} \right| \\
\! &\leq & \left|M \left(\sum_{i,j}^I\frac{\mu^u(B_i)\mu^s(C_j)}{\mu^u(B)\mu^s(C)} - 1 \right) \right| + \left|M \lim_{k \rightarrow \infty}\sum_{i,j>I} \right| \\
\! &<& M\frac{\epsilon}{2M} + M\frac{\epsilon}{2M} \\
\! &=& \epsilon.
\end{eqnarray*}
This holds for all $\epsilon > 0$ so
$$
\lim_{k \rightarrow \infty}\int_{\Sigma}fd\mu^k = \int_{\Sigma}fd\mu
$$
\end{proof}

We now wish to extend this result to the mixing Smale space case. The main tool will be resolving factor maps, and the results in \cite{putnam05}. 

The following proposition allows us to extend the result of lemma \ref{ratiolimit3} to general mixing Smale spaces.

\begin{prop}\label{ratiolimitResolving}
Let $(X,\varphi)$, and $(Y,\psi)$ be mixing Smale spaces, $\pi: Y \rightarrow X$ an almost 1-to-1 (s or u) resolving factor map, and 
suppose the conclusion of lemma \ref{ratiolimit3} holds for $(Y,\psi)$. Then the conclusion of lemma \ref{ratiolimit3} holds for $(X,\varphi)$.
\end{prop}

\begin{proof}  
Suppose $\pi$ is $u$-resolving (the $s$-resolving case is completely analogous).
Let $x_1$, $x_2 \in X$ and $B \subset X^u(x_1)$, $C \subset X^s(x_2)$, and let 
$$
h_X^k = h^k_{B,C}, \ \ 
\mu_X^k = \mu^k_{B,C} .
$$
Now, set $C' = \pi_1^{-1}(C)$. By Lemma \ref{finitepreimage} $ C' = \cup_1^m C'_i$, where the union is disjoint and $C'_i \subset Y^s(y_{2,i})$ for some $y_{2,i} \in Y$. Also, fix $y_1 \in \pi^{-1}(x_1)$, and set $B'$ such that $\pi:B' \rightarrow B$ is a homeomorphism, so $B' \in Y^u(y_1)$. Now 
$$
h^k_{B',C'} = \bigcup_1^m h^k_{B',C'_i}, \ \  \#h^k_{B',C'} = \sum_1^m \# h^k_{B',C'_i}
$$
Notice that since $h_{B',C'}^k \subset Y^u(\varphi^{-k}(y_1))$ and $\pi$ is $u$-resolving, $\pi$ is one-to-one (and hence bijective) on $h_{B',C'}^k$. In other words $\# h^k_{B,C} = \#h^k_{B',C'}$.
Also, recall from prop. \ref{resolvmeas} that 
$$
\mu_X^{u,x_1}(B) = \mu_Y^{u,y_1}(B') , \ {\rm and} \ \mu_X^{s,x_2}(C) = \sum_1^m\mu_Y^{s,y_{2,i}}(C'_i).
$$
Now,
\begin{eqnarray*}
\lim_{k \rightarrow \infty} \lambda^{-2k} \# h^k_{B,C} &=& \lim_{k \rightarrow \infty} \lambda^{-2k} \# h^k_{B',C'} \\ 
&=& \lim_{k \rightarrow \infty} \sum_1^m \lambda^{-2k} \# h^k_{B',C_i'} \\
&=& \sum_1^m \mu_Y^{u,y_1}(B')\mu_Y^{s,y_{2,i}}(C'_i) \\
&=& \mu_X^{u,x_1}(B)\mu_X^{s,x_2}(C)
\end{eqnarray*}
\end{proof}

We are now ready to prove Theorem \ref{ratiolimitSmSp}


\begin{proof}[Proof of Theorem \ref{ratiolimitSmSp}]
As in Cor. 1.4 in \cite{putnam05}, for the mixing Smale space $(X,\varphi)$ we can find another mixing Smale space $(Y,\psi)$ and 
a mixing SFT $(\Sigma,\sigma)$, as well as almost 1-to-1 factor maps $\pi_1:\Sigma \rightarrow Y$, $\pi_2:Y \rightarrow X$ such that 
$\pi_1$ is $s$-resolving and $\pi_2$ is $u$-resolving.  The first conclusion then follows from lemma \ref{ratiolimit3} and 2 applications of Prop. \ref{ratiolimitResolving}.

For the second statement notice that 

$$
\lim_{k \rightarrow \infty} \lambda^{-2k}\# h^k_{B,C} = \mu^u(B)\mu^s(C).
$$

Hence
\begin{eqnarray*}
\lim_{k \rightarrow \infty} \log(\lambda^{-2k}\# h^k_{B,C}) &=& \log(\mu^u(B)\mu^s(C)) \\
\lim_{k \rightarrow \infty} \left (\log(\# h^k_{B,C}) - 2k\log(\lambda) - \log(\mu^u(B)\mu^s(C))  \right ) &=& 0 \\
\lim_{k \rightarrow \infty} \left (\frac{\log(\# h^k_{B,C})}{2k} - h(X, \varphi) - \frac{\log(\mu^u(B)\mu^s(C))}{2k}  \right ) &=& 0 \\
\lim_{k \rightarrow \infty} \left (\frac{\log(\# h^k_{B,C})}{2k} - h(X, \varphi) \right ) &=& 0
\end{eqnarray*}
\end{proof}

The following proposition allows us to extend Prop. \ref{BowenHomSFTMix2}  from the mixing SFT case to the mixing Smale space case and prove Theorem \ref{BowenHomMix}.
\begin{prop}\label{BowenHomResMix}
Let $(X,\varphi)$, and $(Y,\psi)$ be mixing Smale spaces, $\pi: Y \rightarrow X$ an almost 1-to-1 (s or u) resolving factor map, and suppose the conclusion of Theorem \ref{BowenHomMix} holds for $(Y,\psi)$. Then the conclusion of Theorem \ref{BowenHomMix} holds for $(X,\varphi)$.
\end{prop}

\begin{proof}  
Suppose $\pi$ is $u$-resolving (the $s$-resolving case is completely analogous).
Let $x_1$, $x_2 \in X$ and $B \subset X^u(x_1)$, $C \subset X^s(x_2)$, and let 
$$
h_X^k = h^k_{B,C}, \ \ 
\mu_X^k = \mu^k_{B,C} .
$$
Now, set $C' = \pi^{-1}(C)$. By Lemma \ref{finitepreimage} $ C' = \cup_1^m C'_i$, where the union is disjoint and $C'_i \subset Y^s(y_{2,i})$ for some $y_{2,i} \in Y$. Also, fix $y_1 \in \pi^{-1}(x_1)$, and set $B'$ such that $\pi:B' \rightarrow B$ is a homeomorphism, so $B' \in Y^u(y_1)$. Now set 
$$
h_X^k = h^k_{B',C'} = \bigcup_1^m h^k_{B',C'_i}, \ \ 
\mu_X^k = \mu^k_{B',C} = \sum_1^m \frac{\#h^k_{B',C'_i}}{\#h^k_{B',C'}}\mu^k_{B',C'_i} .
$$
Notice that since $h_Y^k \subset Y^u(\varphi^{-k}(y_1))$ and $\pi$ is $u$-resolving, $\pi$ is one-to-one (and hence bijective) on $h_Y^k$. In other words $h^k_X = \pi(h^k_Y)$, and therefore $\mu_Y^k = (\mu_X^k \circ \pi)$. Also recall from Lemma \ref{ratiolimitSmSp} that 
$$
\lim_{k \rightarrow \infty}\frac{\#h^k_{B',C'_i}}{\#h^k_{B',C'}} = \frac{\mu^{u,y_{2,i}}(C'_i)}{\sum_{j=1}^k\mu^{u,y_{2,j}}(C'_j)}.
$$
Now, for $f \in C(X)$
\begin{eqnarray*}
\int_X f d\mu_X &=& \int_{\pi^{-1}(X)} (f \circ \pi) d(\mu_{X}\circ\pi) = \int_Y (f\circ\pi) d\mu_Y \\
&=& \lim_{k \rightarrow \infty} \int_{Y} (f \circ \pi) d\mu_{B',C'_i}^k \ \ \textrm{for any $i$, by hypothesis} \\
&=& \left(\lim_{k \rightarrow \infty} \int_{Y} (f \circ \pi) d\mu_{B',C'_i}^k\right)\sum_{i=1}^m\frac{\mu^u(C'_i)}{\sum_{j=1}^k\mu^{u,y_{2,j}}(C'_j)} \\
&=& \left(\lim_{k \rightarrow \infty} \sum_1^m \frac{\#h^k_{B',C'_i}}{\#h^k_{B',C'}}\int_{Y}(f \circ \pi) d\mu_{B',C'_i}^k\right) \\
&=& \lim_{k \rightarrow \infty} \int_{Y} (f \circ \pi) d\mu_Y^k \\
&=& \lim_{k \rightarrow \infty} \int_{Y} (f \circ \pi) d(\mu_X^k \circ \pi) \\
&=& \lim_{k \rightarrow \infty} \int_{X} f d\mu_X^k
\end{eqnarray*}
\end{proof}

We are now ready to prove Theorem \ref{BowenHomMix}.
\begin{proof}[Proof of Theorem \ref{BowenHomMix}]
As in Cor. 1.4 in \cite{putnam05}, for the mixing Smale space $(X,\varphi)$ we can find 
another mixing Smale space $(Y,\psi)$ and a mixing SFT $(\Sigma,\sigma)$, as well as almost one-to-one
 factor maps $\pi_1:\Sigma \rightarrow Y$, $\pi_2:Y \rightarrow X$ such that $\pi_1$ is $s$-resolving and $\pi_2$ is $u$-resolving.  
The conclusion then follows from Prop. \ref{BowenHomSFTMix2} and 2 applications of Prop. \ref{BowenHomResMix}.
\end{proof}

Finally, we prove Theorems \ref{ratiolimitSmSpIrred} and \ref{BowenHomIrred}. 

\begin{proof}[Proof of Theorems \ref{ratiolimitSmSpIrred} and \ref{BowenHomIrred}]
We assume that $B, C$ are contained in $X_{i_{0}}$. 
Without loss of generality, we assume $i_{0}=1$.
Since for any $n \geq 0$, $\varphi^{n}(B), \varphi^{n}(C)$ are both contained in $X_{1 + n}$ 
(where $1 + n$ is interpreted modulo $I$), the intersection of $h_{B,C}^{k}$ with $X_{i}$ is 
$\varphi^{kI + i - 1}(B) \cap \varphi^{-kI + i - 1}(C)$, which we denote by $h_{i}^{k}$.
Furthermore, we define
\[
 \mu_{i}^{k} = (\# h_{i}^{k})^{-1} \sum_{z \in h_{i}^{k}} \delta_{z}.
\]
 With $ 1 \leq i \leq I$ fixed, consider
Theorem \ref{ratiolimitSmSp} applied to the system $( X _{i}, \varphi^{I} \mid X_{i})$ with local unstable and stable sets $\varphi^{i-1}(B)$ and 
$\varphi^{i-1}(C)$. Notice also that 
$h(X_{i}, \varphi^{I}) = I h(X, \varphi)$, so if $\log(\lambda) = h(X, \varphi)$, $\log(\lambda^I) = h(X_{i}, \varphi^{I})$. It now follows that
\begin{eqnarray*}
 \lim_{k} \#  h^{k}_{i} (\lambda^I)^{-2k} & = &  \mu^{u, \varphi^{i-1}(x)}(\varphi^{i-1}(B)) 
             \mu^{u, \varphi^{i-1}(y)}(\varphi^{i-1}(C))  \\
             &  =  &  \lambda^{ 1 - i} \mu^{u,x}(B) \lambda^{i - 1} \mu^{s,y}(C) \\
              & =  &  \mu^{u,x}(B)  \mu^{s,y}(C).
\end{eqnarray*}
Noticing that 
$$
\lim_k \frac{\#h^{k}_{B,C}}{\# h^{k}_{i}} = I
$$
we have
$$
\lim_{k} \#  h^{k}_{B,C} \lambda^{-2kI} = I\mu^{u,x}(B)  \mu^{s,y}(C).
$$
It then follows as in the proof of Theorem \ref{ratiolimitSmSp} that
$$
\lim_{k} \frac{h^{k}_{B,C}}{2kI} = h(X,\varphi).
$$
We also note that, Theorem \ref{BowenHomMix} implies 
\[
 \lim_{k} \mu^{k}_{i} = \mu_{X_{i}}.
\]
Putting all of this together, we have
\begin{eqnarray*}
 \lim_{k} \mu_{B,C}^{k} & = & \lim_{k} (\# h^{k}_{B,C})^{-1} \sum_{z \in h^{k}_{B,C}} \delta_{z}  \\
       &  =  &  \lim_{k}  (\# h^{k}_{B,C})^{-1} \sum_{i=1}^{I}  \sum_{z \in h^{k}_{i}} \delta_{z}  \\ 
       &  =  &  \lim_{k}  \sum_{i=1}^{I}  \frac{\# h^{k}_{i}}{\# h^{k}_{B,C}} (\# h^{k}_{i})^{-1} \sum_{z \in h^{k}_{i}} \delta_{z} \\
   &  =  &   \lim_{k}  \sum_{i=1}^{I}  \frac{\# h^{k}_{i}}{\sum_{j=1}^{I} \# h^{k}_{j}} \mu^{k}_{i} \\
    &  =  &   \lim_{k} \sum_{i=1}^{I}  \frac{\# h^{k}_{i} \lambda^{-2kI}}{ \sum_{j=1}^{I} \# h^{k}_{j} \lambda^{-2kI} }   \mu^{k}_{i} \\
    &  =  &   \sum_{i=1}^{I} \frac{ \mu^{u,x}(B) \mu^{s,y}(C)}{ I \mu^{u,x}(B) \mu^{s,y}(C)} \mu_{X_{i}} \\
  &  =  &  \sum_{i=1}^{I} \frac{1}{I} \mu_{X_{i}} \\
  &  =  &  \mu_{X}.
\end{eqnarray*}
\end{proof}



\end{document}